\documentclass[11pt]{amsart}
\usepackage{mathptmx}
\usepackage{amscd,amssymb,amsmath,amsthm}

\newtheorem{theorem}[equation]{Theorem}
\newtheorem{prop}[equation]{Proposition}
\newtheorem{lemma}[equation]{Lemma}
\newtheorem{cor}[equation]{Corollary}

\newtheorem{question}[equation]{Question}

\theoremstyle{remark}

\theoremstyle{definition}

\newtheorem*{CNT}{Courant's nodal domain theorem}
\newtheorem*{EI}{The Euler inequality}
\numberwithin{equation}{subsection}

\theoremstyle{remark}
\newtheorem*{remark*}{Remark}
\newtheorem*{example*}{Example}
\newtheorem*{quest*}{Question}

\DeclareMathAlphabet{\matheur}{U}{eur}{m}{n}
\DeclareMathAlphabet{\matheus}{U}{eus}{m}{n}
\DeclareMathAlphabet{\matheuf}{U}{euf}{m}{n}

\newcommand{\abs}[1]{\left\lvert#1\right\rvert}

\DeclareMathOperator{\ord}{ord}

\author{Mikhail Karpukhin}
\address{Department of Geometry and Topology, Moscow State University, Leninskie Gory, GSP-1, 119991, Moscow, Russia}
\address{Independent University of Moscow, Bolshoy Vlasyevskiy pereulok
11, 119002, Moscow, Russia}
\email{karpukhin@mccme.ru}
\author{Gerasim  Kokarev}
\address{Mathematisches Institut der Universit\"at M\"unchen, Theresienstr. 39, D-80333 M\"unchen, Germany}
\email{Gerasim.Kokarev@mathematik.uni-muenchen.de}
\author{Iosif Polterovich}
\address{D\'epartement de math\'ematiques et de statistique, Universit\'e de Montr\'eal, CP 6128 succ Centre-Ville, Montr\'eal, QC H3C 3J7, Canada}
\email{iossif@dms.umontreal.ca}

\title{Multiplicity bounds for Steklov eigenvalues on Riemannian surfaces}

\begin{document}

\begin{abstract} 
We prove two  explicit  bounds for the multiplicities of Steklov eigenvalues $\sigma_k$ on  compact surfaces with boundary.  One of the bounds depends only  on the genus of a surface and the index $k$ of an eigenvalue, while the other depends as well on  the number of boundary components. We also show that on any given Riemannian surface with smooth boundary the multiplicities of Steklov eigenvalues $\sigma_k$ are uniformly bounded in $k$. 
\end{abstract}
\maketitle


\section{Introduction and main results}
\subsection{Multiplicity bounds for Laplace eigenvalues}
Let $M$ be a smooth closed surface. For a Riemannian metric $g$ on $M$ we denote by
$$
0=\lambda_0(g)<\lambda_1(g)\leqslant\ldots\lambda_k(g)\leqslant\ldots
$$
the eigenvalues of the Laplace operator $-\Delta_g$. A classical result by Cheng in~\cite{Cheng} says that the multiplicities $m_k(g)$ of these eigenvalues are bounded by quantities depending on the genus $\gamma$ of $M$ only. Cheng's bound was sharpened by Besson~\cite{Be} for orientable surfaces, and by Nadirashvili~\cite{Na} in the general case, to the following estimate for multiplicities:
\begin{equation}
\label{laplace}
m_k(g)\leqslant 2(2-\chi)+2k+1,\qquad  k=1,2\ldots,
\end{equation}
where $\chi$ is the Euler-Poincar\'e number of $M$. If $M$ is  homeomorphic to a sphere or a projective plane, inequality~\eqref{laplace} is sharp for $k=1$.

The purpose of this paper is  to prove multiplicity bounds  for boundary value problems on Riemannian surfaces.   We are essentially concerned with the Steklov eigenvalue problem, for which multiplicity bounds are known only in the case of simply connected domains, see~\cite{Al}.  For the Dirichlet and Neumann boundary value problems the multiplicity bounds are due to~\cite{Na, HoNa}, where the authors also consider simply connected domains only. At the end of the paper we discuss versions of these results for arbitrary Riemannian surfaces with boundary, as well as for more general eigenvalue problems.
\subsection{Steklov eigenvalue problem}
From now on let  $(M,g)$ be a smooth compact Riemannian surface with a non-empty boundary. For a given bounded non-negative function $\rho$ on the boundary $\partial M$ the Steklov eigenvalue problem is  stated as:
\begin{equation}
\label{steklov}
\Delta_g u=0\quad\text{ in }M,\quad\text{and}\quad \frac{\partial u}{\partial\nu}=\sigma\rho u\quad\text{ on }\partial M,
\end{equation}
where $\nu$ is an outward unit normal. Denote by $\mu$ an absolutely continuous  measure on the boundary $\partial M$ with the density $\rho$, that is $d\mu=\rho\mathit{ds}_g$. The real numbers $\sigma$ for which a  nonzero  solution above exists are  eigenvalues of the Dirichlet form $\int\abs{\nabla u}^2\mathit{dVol}_g$ in the space $L_2(M,\mu)$. Its spectrum is non-negative and discrete, see~\cite{Ba}, and  we denote by
$$
0=\sigma_0(g,\mu)<\sigma_1(g,\mu)\leqslant\ldots\sigma_k(g,\mu)\leqslant\ldots
$$
the corresponding eigenvalues. This eigenvalue problem was considered in $1902$ by Steklov and since then has been studied extensively; we refer to~\cite{Ba} and the recent papers~\cite{GP,FS} for a comprehensive list of references on the subject.
\subsection{Main results} 
Our main result is the following theorem.
\begin{theorem}
\label{t1}
Let $(M,g)$ be a compact Riemannian surface with a non-empty boundary, and $\mu$ be an absolutely continuous Radon measure on $\partial M$ whose density is bounded. Then the multiplicity $m_k(g,\mu)$ of the Steklov eigenvalue $\sigma_k(g,\mu)$ satisfies the inequalities
\begin{equation}
\label{in1}
m_k(g,\mu)\leqslant 2(2-\bar\chi)+2k+1,
\end{equation}
\begin{equation}
\label{in2}
m_k(g,\mu)\leqslant 2(2-\bar\chi)+2l+k,  
\end{equation}
for all $k=1,2\dots$,  where $\bar\chi=\chi+l$, and $\chi$ and $l$ stand for the Euler-Poincar\'e number and the number of boundary components of $M$ respectively. Besides, inequality~\eqref{in2} is strict for an even $k$.
\end{theorem}
Note that $\bar\chi$ depends  on the genus $\gamma$  of $M$ only. More precisely, it equals $2-2\gamma$ for orientable surfaces and  $2-\gamma$ for non-orientable ones. Both inequalities above are similar to the Besson-Nadirashvili multiplicity bounds on closed Riemannian surfaces. The right hand-side of~\eqref{in1}  is the same function of the genus of $M$ as in~\eqref{laplace}. This bound does not depend on any boundary data and, as we show in Section~\ref{other}, holds for other boundary value problems. The second inequality can be re-written in the form
$$
m_k(g,\mu)\leqslant 2(2-\chi)+k.
$$
It is specific to the Steklov problem, and for $k\geqslant 2l$ is sharper than the first one. The proofs of both inequalities  are built on the ideas due to~\cite{Na, HoNa} and use the properties of nodal graphs. In comparison with other classical boundary  value problems,  there is an additional difficulty related to the fact that there is no known local model for the nodal set  of  a Steklov eigenfunction at the boundary points. 
In particular,  one has to show that the nodal graph is finite, see Lemma~\ref{l1}.

The statement that inequality~\eqref{in2} is strict for an even $k$ is a consequence of our method, see Section~\ref{proofs}. Under an additional topological hypothesis on a surface $M$, it is strict for any $k\geqslant 1$.
\begin{theorem} 
\label{t2}
Under the hypotheses of Theorem~\ref{t1}, suppose that $M$ is not homeomorphic to a disk. Then inequality~\eqref{in2} is strict for any $k\geqslant 1$.
\end{theorem}

\noindent
The proof of Theorem~\ref{t2} is based on the careful analysis of the equality case; it appears in Section~\ref{proofs2}. The key ingredient is an isotopy argument for nodal graphs, similar to the one in~\cite{HoNa1,HoNa}.

\subsection{Discussion}
For a disk, Alessandrini and Magnanini proved in~\cite{Al} the bound $m_k(g,\mu)\leqslant 2k$, which is sharp for the first eigenvalue. In comparison, our inequality \eqref{in2} shows that $m_k(g,\mu)\leqslant k+2$ for an odd $k$ and $m_k(g,\mu)\leqslant k+1$ for an even~$k$. 

We emphasise two cases when our results give sharp bounds, as follows from the results in~\cite{FS, FS1}.
\begin{cor} 
\label{sharp}
Under the hypotheses of Theorem~\ref{t1},
\begin{itemize}
\item [(i)] if $M$ is homeomorphic to an orientable surface of zero genus with $l \geqslant 2$ boundary components, then the multiplicity of the first non-zero Steklov eigenvalue is at most three;
\item [(ii)] if $M$ is  homeomorphic to a M\"obius band, then the multiplicity of the first non-zero Steklov eigenvalue is at most four.
\end{itemize}
\end{cor}
For a M\"obius band ($\chi=0$, $l=1$) the inequality  $m_1(g,\mu)\leqslant 4$ follows from Theorem~\ref{t2}. The equality is attained at a ``critical'' M\"obius band explicitly described in~\cite{FS1}. For an annulus ($\chi=0$, $l=2$)  the inequality $m_1(g,\mu)\leqslant 3$ follows from \eqref{in1}. The equality is attained at a ``critical'' catenoid constructed in \cite{FS}.  More generally, as was recently shown in \cite{FS1}, on any orientable surface of zero genus with $l\geqslant 2$ boundary components,  there exists a metric admitting a minimal embedding to a $3$-dimensional unit ball by first Steklov eigenfunctions. In particular, the first non-zero Steklov eigenvalue of such a metric has multiplicity three, and inequality~\eqref{in1} is also sharp on these surfaces.

 In general, Theorems \ref{t1} and \ref{t2} do not  give sharp multiplicity bounds. It is an interesting question to understand whether $m_k(g,\mu)$ is uniformly bounded in all parameters; see Section \ref{asbo}. More specifically, one may ask the following question, cf.~\cite[Question 1.8]{GP1}:
\begin{question} Does  there exist a sequence of surfaces $(M_n,g_n)$ with boundary measures $\mu_n$ such that $m_1(g_n,\mu_n) \to \infty$ as $n\to +\infty$ ?
\end{question}
If such a sequence exists, by inequality~\eqref{in1} the corresponding genera $\gamma_n$ of $M_n$ tend to infinity. Note also that  the answer to an analogous question  for the multiplicity of the first  Laplace eigenvalue is positive~\cite{CB, CdV}. 

\medskip
\noindent
{\em Remark.} While the present paper was at the final stage of preparation, a  different proof of  inequality~\eqref{in1} for orientable surfaces (and, consequently, of part (i) of Corollary~\ref{sharp}) appeared in~\cite{FS1,Jammes}. The approaches behind all the proofs go back to the ideas of Cheng and Besson.  At the same time, our proof that the nodal graph is finite is different from the one in~\cite{FS1}: it is based on a topological argument and uses only Courant's nodal domain theorem. Besides, it applies to general boundary measures $\mu$, see Lemma~\ref{l1} and the discussion in Section~\ref{other}. Note also that for non-orientable surfaces, inequality~\eqref{in1} is sharper than the bound in~\cite{FS1,Jammes}.
\subsection{Asymptotic bounds for Steklov eigenvalues}
\label{asbo}
Suppose that the boundary $\partial M$ is smooth, and the weight function $\rho$ in~\eqref{steklov} is smooth and strictly positive. Then the Steklov eigenvalues can be viewed as the eigenvalues of a self-adjoint elliptic pseudo-differential operator of the first order; it sends a function on $\partial M$ to the normal derivative of its harmonic extension multiplied by $\rho^{-1}$.  In particular,  for $\rho \equiv 1$ this pseudo-differential operator is precisely  the Dirichlet-to-Neumann operator on $\partial \Omega$.  Using H\"ormander's theorem on spectral asymptotics for pseudo-differential operators \cite{Ho}, we obtain the following result.
\begin{theorem} 
\label{as1}
Let $(M,g)$ be a compact Riemannian surface with a smooth boundary and $\mu$ be a measure on $\partial M$ whose density $\rho$ is smooth and strictly positive. Then the multiplicities $m_k(g,\mu) $ of Steklov eigenvalues are uniformly bounded in $k$, i.e. there exists a constant $C_{g,\mu}$, depending on a metric $g$ and a measure
$\mu$,  such that 
$$
m_k(g,\mu)\leqslant C_{g,\mu}\qquad\text{ for all}\quad k=1,2,\dots.
$$ 
\end{theorem}
The version of this result for Laplace eigenvalues is well-known, see \cite{HoNa1}; in that case Weyl's law with a sharp remainder estimate implies that $m_k(g)=O(\sqrt{k})$ as $k\to +\infty$.

When $M$ is a disk, Theorem~\ref{as1} can be strengthened to the following statement.
\begin{prop} 
\label{as2}
Under the hypotheses of Theorem~\ref{as1}, suppose that $M$ is homeomorphic to a disk. Then there exists an integer $K_{g,\mu}>0$, depending on a metric $g$ and a measure $\mu$, such that $m_k(g, \mu)\leqslant 2$ for all $k>K_{g,\mu}$.
\end{prop}
Note that the inequality above is sharp; it is attained on a Euclidean disk. The proof of Proposition~\ref{as2} uses the uniformisation theorem and the sharp asymptotics for the Steklov eigenvalues of a Euclidean disk ~\cite{Ros, Ed} .

\section{Preliminaries}
\label{prems}
\subsection{Variational principle and Courant's nodal domain theorem}
We start with recalling  a variational setting for the Steklov eigenvalue problems. Given a Riemannian surface $(M,g)$ and a measure $\mu$ on its boundary, the Steklov eigenvalues can be defined by the min-max principle
$$
\sigma_k(g,\mu)=\inf_{\Lambda^{k+1}}\sup_{u\in\Lambda^{k+1}}\matheur R_g(u,\mu),
$$
where the infimum is taken over all $(k+1)$-dimensional subspaces $\Lambda^{k+1}\subset L_2(M,\mu)$ formed by $C^\infty$-smooth functions, the supremum is  over all nonzero  $u\in\Lambda^{k+1}$, and $\matheur R_g(u,\mu)$ stands for the Rayleigh quotient
$$
\matheur R_g(u,\mu)=\left(\int_M\abs{\nabla u}^2\mathit{dVol}_g\right)/\left(\int_Mu^2 d\mu\right).
$$
Here we view $\mu$ as measure on $M$ supported on the boundary $\partial M$. The Steklov eigenfunctions can be then regarded as solutions of the equation
$$
\int_M\langle\nabla u,\nabla\varphi\rangle\mathit{dVol}_g=\sigma_k(g,\mu)\int_Mu\varphi d\mu
$$
understood as an integral identity, where $\varphi$ is a $C^\infty$-smooth test-function. The equation above can be also viewed as a Schr\"odinger equation whose potential is a measure supported on the boundary of $M$.

Let $u$ be a Steklov eigenfunction. It is harmonic inside $M$, and, in particular, is $C^\infty$-smooth. By $\mathcal N(u)$ we denote its nodal set, that is the set $u^{-1}(0)$. Recall that a connected component of $M\backslash\mathcal N(u)$ is called its nodal domain. By maximum principle, it is straightforward to conclude that the closure of each nodal domain has a non-trivial intersection with the boundary $\partial M$. Further, by the strong maximum principle~\cite{GT}, any Steklov eigenfunction has different signs on adjacent nodal domains. For the sequel we need a version of Courant's nodal domain theorem for Steklov eigenfunctions. 
\begin{CNT}
Let $(M,g)$ be a compact Riemannian surface with boundary, and $\mu$ be an absolutely continuous Radon measure on $\partial M$ whose density is bounded. Then each Steklov eigenfunction $u$ corresponding to the eigenvalue $\sigma_k(g,\mu)$ has at most $(k+1)$ nodal domains.
\end{CNT}
The proof of this theorem uses the min-max principle and is similar to the one for Laplace eigenfunctions.  For the Steklov problem on planar domains it can be found in~\cite{KS}, and the argument holds  for arbitrary Riemannian surfaces with boundary.

\subsection{Local behaviour of harmonic functions; vanishing order}
Let $u$ be a harmonic function on $M$ and $x\in M$ be an interior point. The {\em vanishing order} of $u$ at $x$ is a non-negative integer, denoted by $\ord_x(u)$, that is the order of the first non-vanishing derivative of $u$ at $x$. The following statement is classical, see~\cite{Bers} and~\cite[Theorem~4.1]{Helff}, and holds for solutions of rather general second order linear elliptic equations. 
\begin{prop}
\label{BersTheorem}
Let $(M,g)$ be a compact Riemannian surface with boundary, and $u$ be a harmonic function on $M$. Then for any interior point $x_0\in M$ there exist its neighbourhood chart $U$ and a non-trivial homogeneous harmonic polynomial $P_n$ of degree $n=\ord_{x_0}(u)$ on the Euclidean plane $\mathbf R^2$ such that 
$$
u(x) = P_n(x-x_0) + O(\abs{x-x_0}^{n+1}),
$$ 
where $x\in U$.
\end{prop}
In the proposition above we assume that the neighbourhood $U$ is such that the metric $\left.g\right|_U$ is conformally Euclidean. In particular, the property of being harmonic on $U$ with respect to the metric $g$ is equivalent to being harmonic with respect to the Euclidean metric. Now for a given positive integer $\ell$ consider the set
$$
\mathcal N^\ell(u)=\{x\in M ~|~ \ord_x(u)\geqslant \ell\}.
$$ 
Using Proposition~\ref{BersTheorem}, in~\cite{Cheng} Cheng shows that around a point $x_0\in\mathcal N(u)$  the nodal set is diffeomorphic to the nodal set of the corresponding harmonic polynomial $P_n$, which consists of $n=\ord_{x_0}(u)$ lines meeting at the origin. In particular, the set $\mathcal N^2(u)$ consists of isolated points in the interior of $M$, and  the complement $\mathcal N^1(u)\backslash\mathcal N^2(u)$ is a collection of $C^\infty$-smooth arcs. Thus, the nodal set $\mathcal N(u)$ can be viewed as a graph in the interior of $M$ whose vertices are points $x\in\mathcal N^2(u)$ and edges are connected components of $\mathcal N^1(u)\backslash\mathcal N^2(u)$. In the sequel we refer to $\mathcal N(u)$ as the {\em nodal graph}, meaning this graph structure. 

It is not hard  to construct  harmonic functions on compact surfaces with boundary whose nodal graphs are infinite.  
One way to ensure the finiteness of the nodal graph of a harmonic function is to impose certain regularity on its boundary behaviour, see~\cite{Al1}. For Steklov eigenfunctions, we adopt an approach based on Courant's nodal domain theorem. The following statement  is a direct consequence of Lemma~\ref{l1} in Section~\ref{proofs}.
\begin{prop}
\label{p1}
Let $(M,g)$ be a compact Riemannian surface with boundary, and $\mu$ be an absolutely continuous Radon measure on $\partial M$ whose density is bounded. Then the nodal graph $\mathcal N(u)$ of a non-trivial Steklov eigenfunction $u$ has a finite number of vertices and edges.
\end{prop}

\subsection{Graphs in surfaces: basic background}
The purpose of this subsection is to introduce notation and collect a number of auxiliary facts used throughout the rest of the paper. Let $S$ be a surface, possibly non-compact. Recall that a graph $\Gamma\subset S$ is a collection of points, called {\em vertices},  and embedded open intervals, called {\em edges}, such that the boundary of each edge belongs to the set of vertices. In addition, we assume that edges do not intersect and do not contain vertices. A graph is called {\em compact} if it is compact as a subset; it is called {\em finite}, if it has a finite number of vertices and edges. For example, for a non-trivial Steklov eigenfunction $u$ the nodal graph $\mathcal N(u)$, viewed as a subset in the interior of $M$, is not compact, since it contains edges approaching the boundary.

Let $\Gamma$ be a finite compact graph in $S$. For a vertex $x\in\Gamma$ its {\em degree} $\deg_\Gamma(x)$ is the number of edges incident to  $x$; if there is an edge that starts and ends at $x$, then it counts twice. The number of edges $e$ of a finite compact graph is given by the formula
\begin{equation}
\label{edges}
2e=\sum\deg_\Gamma(x),
\end{equation}
where the sum is taken over all vertices $x\in\Gamma$. Connected components of $S\backslash\Gamma$ are called {\em faces} of $\Gamma$. The following inequality is a consequence of the standard Euler formula for a cell complex, see~\cite[p.~207]{Gib}.
\begin{EI}
Let $\Gamma$ be a finite graph in a closed surface $S$, and $v$, $e$, and $f$ be the number of its vertices, edges, and faces respectively. Then the following inequality holds:
\begin{equation}
\label{euler}
v-e+f\geqslant \chi,
\end{equation}
where $\chi$ is the Euler-Poincar\'e number of $S$. Besides, the equality occurs if and only if $\Gamma$ is the $1$-skeleton of a cell decomposition of $S$.
\end{EI}
We end with recalling the terminology for paths in graphs, which is used at the end of Sect.~\ref{proofs}. By a path in a graph $\Gamma$ we mean a continuous map $\phi:[0,1]\to\Gamma$ such that $\phi(0)$ and $\phi(1)$ are vertices, and if the image of $\phi$ intersects non-trivially with an edge, then it contains this edge. A path in $\Gamma$ is called {\em simple}, if it has no repeated vertices and edges. A closed path in a finite graph is called the {\em simple cycle}, or {\em circuit}, if it has no repeated vertices and edges except for $\phi(0)=\phi(1)$. A {\em tree} is a connected graph that has no circuits; its every two vertices can be joined by a simple path. 

Finally, mention that a finite graph $\Gamma$ in a closed surface whose vertices have degree at least two always contains a circuit.

\section{Proof of Theorem~\ref{t1}}
\label{proofs}
\subsection{Reduced nodal graph}
Let $M$ be a smooth Riemannian surface with a non-empty boundary and $\bar M$ be a closed surface of the same genus, viewed as the image of $M$ under collapsing its boundary components to points. By $\bar{\mathcal N}(u)$ we denote the corresponding image of a nodal graph $\mathcal N(u)$; we call it the {\em reduced nodal graph}. More precisely,  its edges are the same nodal arcs, and there are two types of vertices: vertices that correspond to the boundary components that contain limit points of nodal lines, referred as {\em boundary component vertices}, and genuine vertices that correspond to the points in $\mathcal N^2(u)$, referred as {\em interior vertices}. It is straightforward to see that the number of nodal domains of an eigenfunction $u$ is precisely the number of the connected components of $\bar M\backslash\bar{\mathcal N}(u)$. Throughout the paper we use the notation $\bar\chi$ for the Euler-Poincar\'e number of $\bar M$. It coincides with the quantity $\chi+l$,  used in Theorem~\ref{t1}, and is called the {\em reduced Euler-Poincar\'e number} of $M$.

The following lemma is a basis for the proof of Theorem~\ref{t1}. It uses only Courant's nodal domain theorem, and holds for eigenfunctions of rather general boundary value problems.
\begin{lemma}
\label{l1}
Let $(M,g)$ be a compact Riemannian surface with boundary, and $\mu$ be an absolutely continuous Radon measure on $\partial M$ whose density is bounded. Then the reduced nodal graph $\bar{\mathcal N}(u)$ of a non-trivial Steklov eigenfunction $u$ is finite, i.e. it has a finite number of vertices and edges.
\end{lemma}
\begin{proof}
Consider the reduced nodal graph $\bar{\mathcal N}(u)$ corresponding to a non-trivial Steklov eigenfunction $u$. For a proof of the lemma it is sufficient  to rule out the occurrence of:
\begin{itemize}
\item[(i)] boundary component vertices of infinite degree and 
\item[(ii)] the infinite number of interior vertices 
\end{itemize}
in $\bar{\mathcal N}(u)$. We are going to construct new graphs in $\bar M$ by resolving interior vertices of $\bar{\mathcal N}(u)$ in the following fashion. Let $x\in\mathcal N^2(u)$ be an interior vertex; its degree equals $2n$, where $n=\ord_x(u)$. Let $U$ be 
a small disk centered at $x$ that does not contain other vertices and such that nodal arcs incident to $x$ intersect $\partial U$ at $2n$ points precisely; the existence of such a disk follows from  Proposition \ref{BersTheorem}. We denote these intersection points by $y_i$, where $i=0,\ldots, 2n-1$, and assume that they are ordered consequently in the clockwise fashion. A new graph is obtained from $\bar{\mathcal N}(u)$ by changing it inside $U$ and removing possibly appeared edges without vertices. More precisely, we remove the nodal set inside $U$ and round-off the edges on the boundary $\partial U$ by non-intersecting arcs in $U$ joining the points $y_{2j}$ and $y_{2j+1}$. If there was an edge that starts and ends at $x$, then such a procedure may make it into a loop. If this occurs, then we remove this loop to obtain a genuine graph in $\bar M$. It has one vertex less and at most as many faces as the original graph.

\noindent
{\em Ruling out~(i).} Let us resolve each interior vertex in $\bar{\mathcal N}(u)$ in the way described above. The result is a graph $\Gamma$ in $\bar M$ whose only vertices are boundary component vertices in $\bar{\mathcal N}(u)$; we denote by $v$ their number. Besides, it has at most as many faces as the reduced nodal graph, that is by Courant's nodal domain theorem at most $k+1$. Suppose that the reduced nodal graph has a boundary component vertex of infinite degree; then so does $\Gamma$. Let us remove all edges in $\Gamma$ except for $v+k+2-\bar\chi$ of them to obtain a new finite graph, and denote by $f$ the number of its faces. Since removing an edge does not increase the number of faces, we have $f\leqslant k+1$. On the other hand, by the Euler inequality~\eqref{euler}, we have
$$
f\geqslant e-v+\bar\chi=k+2.
$$
Thus, we arrive at a contradiction.

\noindent
{\em Ruling out~(ii).} Suppose the contrary; the situation described in~(ii) occurs. Let $v$ be the number of boundary component vertices in $\bar{\mathcal N}(u)$. Let us resolve all interior vertices except for $v+k+2-\bar\chi$ of them. The result is a finite graph $\Gamma'$. Denote by $v'$, $e'$, and $f'$ the number of its vertices, edges, and faces respectively; then we have
$$
v'\leqslant 2v+k+2-\bar\chi\quad\text{and}\quad e'\geqslant 2(v+k+2-\bar\chi).
$$
Here in the second inequality we used formula~\eqref{edges} and  the fact that the degree of each vertex $x\in\mathcal N^2(u)$ is at least 4. Combining these two inequalities with the Euler inequality~\eqref{euler}, we obtain
$$
f'\geqslant e'-v'+\bar\chi\geqslant k+2.
$$
Thus, we arrive at a contradiction with Courant's nodal domain theorem.
\end{proof}

\subsection{Multiplicity bounds: the first inequality}
We start with a lemma that gives a lower bound for the number of nodal domains via the vanishing order of points $x\in\mathcal N^2(u)$. For the Dirichlet boundary problem on surfaces of zero genus it is proved in~\cite{HoNa}. We give a rather simple proof based on the use of the Euler inequality. 
\begin{lemma}
\label{l3}
Let $(M,g)$ be a compact Riemannian surface with boundary, and $\mu$ be an absolutely continuous Radon measure on $\partial M$ whose density is bounded. Then for any non-trivial Steklov eigenfunction $u$ the number of its nodal domains is at least $\sum(\ord_x(u)-1)+\bar\chi$, where the sum is taken over all points in $\mathcal N^2(u)$ and $\bar\chi$ is the reduced Euler-Poincar\'e number of $M$.
\end{lemma}
\begin{proof}
Let $\bar{\mathcal N}(u)$ be a reduced nodal graph in $\bar M$, and $v$, $e$, and $f$ be the number of its vertices, edges, and faces respectively; by $r$ we denote the number of boundary component vertices. Using formula~\eqref{edges}, we get
$$
e\geqslant r+\sum\ord_x(u),
$$
where the sum is taken over $x\in\mathcal N^2(u)$. Here we used the fact that the degree of each boundary component vertex is at least two. Viewing $v$ as the sum $r+\sum 1$, where the sum symbol is again over $x\in\mathcal N^2(u)$, by the Euler inequality we obtain
$$
f\geqslant e-v+\bar \chi\geqslant\sum(\ord_x(u)-1)+\bar\chi.
$$
Since $f$ is precisely the number of nodal domains, we are done.
\end{proof}
The following lemma is a version of the statement due to~\cite{Na}.
\begin{lemma}
\label{l4}
Let $(M,g)$ be a compact Riemannian surface with boundary, and $u_1,\ldots,u_{2n}$ be a collection of non-trivial linearly independent harmonic functions on $M$. Then for a given interior point $x\in M$ there exists a non-trivial linear combination $\sum\alpha_iu_i$ whose vanishing order at the point $x$ is at least $n$.
\end{lemma}
\begin{proof}
Let $V$ be the span of $u_1,\ldots, u_{2n}$, and $V_i$ be its subspace formed by harmonic functions $u\in V$ whose vanishing order at $x$ is at least $i$, $\ord_x(u)\geqslant i$. Clearly, the subspaces $V_i$ form a nested sequence, $V_{i+1}\subset V_i$. The statement of the lemma says that $V_n$ is non-trivial. Suppose the contrary, that is $V_n=\{0\}$. Then the dimension of $V$ satisfies the relation
$$
\dim V\leqslant 1+\sum_{i=1}^{n-1}\dim (V_i/V_{i+1}).
$$
By Proposition~\ref{BersTheorem} the factor-space $V_i/V_{i+1}$ can be identified with a subspace of homogeneous harmonic polynomials of order $i$. In polar coordinates on $\mathbf R^2$ such polynomials have the form
$$
P_i(r\cos\theta,r\sin\theta)=ar^i\cos(i\theta)+br^i\sin(i\theta);
$$
in particular, they form a space of dimension two. Thus, we obtain
$$
\dim V\leqslant 1+2(n-1)=2n-1.
$$
This is a contradiction with the hypotheses of the lemma.
\end{proof}

Now we prove the first inequality in Theorem~\ref{t1}:
$$
m_k(g,\mu)\leqslant 2(2-\bar\chi)+2k+1. 
$$
Suppose the contrary to its statement. Then there exists at least $2(2-\bar\chi)+2k+2$ linearly independent eigenfunctions corresponding to the eigenvalue $\sigma_k(g,\mu)$. Pick an interior point $x\in M$. By Lemma~\ref{l4} there exists a new eigenfunction $u$ whose vanishing order at the point $x$ is at least $2-\bar\chi+k+1$. Combining this with Lemma~\ref{l3}, we conclude that the number of the nodal domains of $u$ is at least $k+2$. Thus, we arrive at a contradiction with Courant's nodal domain theorem. \qed

\subsection{Multiplicity bounds: the second inequality}
\label{2nd}
The proof of the second inequality is based on the following lower bound for the number of nodal domains.
\begin{lemma}
\label{l2}
Let $(M,g)$ be a compact Riemannian surface with boundary, and $\mu$ be an absolutely continuous Radon measure on $\partial M$ whose density is bounded. Then for any non-trivial Steklov eigenfunction $u$ the number of its nodal domains is at least 
$$
\max\{2\ord_x(u)+2\bar\chi-2l-2~|~ x\in\mathcal N^2(u)\},
$$
where $\bar\chi$ is the reduced Euler-Poincar\'e number of $M$ and $l$ is the number of boundary components.
\end{lemma}
We proceed with the proof of the second inequality in Theorem~\ref{t1}:
$$
m_k(g,\mu)\leqslant 2(2-\bar\chi)+2l+k.
$$
Suppose the contrary. Then there exists at least $2(2-\bar\chi)+2l+k+1$ linearly independent eigenfunctions corresponding to the eigenvalue $\sigma_k(g,\mu)$. Pick an interior point $x\in M$. By Lemma~\ref{l4} there exists a new eigenfunction $u$ whose vanishing order at the point $x$ is at least 
$2-\bar\chi+l+[(k+1)/2]$, where $[\cdot]$ denotes the integer part. Using the estimate in Lemma~\ref{l2}, we see that the number of the nodal domains of $u$ is at least $k+2$. Thus, we arrive at a contradiction with Courant's nodal domain theorem. The same argument shows that this multiplicity bound is strict for an even $k$. \qed

The rest of the section is concerned with the proof of Lemma~\ref{l2}. It is based on the study of certain subgraphs in the nodal graph, which we introduce now. 

For a given vertex $x\in\mathcal N^2(u)$ we denote by $\Gamma_1$ a subgraph of $\mathcal N(u)$  that is the union of all circuits in the connected component of $x$ and all simple paths joining $x$ and the vertices of these circuits. Further, let $\Gamma_2$ be a subgraph of $\mathcal N(u)$ formed by {all simple paths in the nodal set starting from $x$ and approaching the boundary $\partial M$ that do not intersect $\Gamma_1$ except for $x$}. Clearly, the subgraph $\Gamma_2$ does not contain any circuits, and hence, it is a tree. Besides, any nodal edge incident to $x$ belongs either to $\Gamma_1$ or $\Gamma_2$, that is
\begin{equation}
\label{a2}
2\ord_x(u)=\deg_{\Gamma_1}(x)+\deg_{\Gamma_2}(x).
\end{equation}
Finally, we denote by $\Gamma$ the union of $\Gamma_1$ and $\Gamma_2$.

We proceed with the following lemma, which is specific to the Steklov eigenvalue problem.
\begin{lemma}
\label{l22}
Let $u$ be a non-trivial Steklov eigenfunction and  $x\in\mathcal N^2(u)$ be a vertex in its nodal graph. Then the degree of $x$ in $\Gamma_1$ is at most $2l+2-2\bar\chi$. 
\end{lemma}
\begin{proof}
Let $v_1$, $e_1$, and $f_1$ be the number of vertices, edges, and faces of $\Gamma_1$ respectively. Since every vertex in $\Gamma_1$, different from $x$, belongs either to a circuit or the interior of a simple path, its degree in $\Gamma_1$ is at least $2$. Thus, by formula~\eqref{edges} we have
\begin{equation}
\label{a1}
2e_1\geqslant\deg_{\Gamma_1}(x)+2(v_1-1).
\end{equation}
Recall that every nodal domain of $u$ has a non-trivial arc on the boundary. Each face of $\Gamma_1$ contains the union of nodal domains, and therefore it contains at least one boundary component of $M$. Since any two faces of $\Gamma_1$ can not contain the same boundary component, we have $f_1\leqslant l$. Viewing $\Gamma_1$ as a subgraph in the reduced nodal graph $\bar{\mathcal N}(u)$, we can apply the Euler inequality  to obtain
$$
e_1\leqslant v_1+f_1-\bar\chi\leqslant v_1+l-\bar\chi.
$$
Now the statement follows by the combination of this inequality with relation~\eqref{a1}.
\end{proof}

\begin{proof}[Proof of Lemma~\ref{l2}]
Let $x\in\mathcal N^2(u)$ be a vertex in the nodal graph. Consider a subgraph $\Gamma_2$ of the nodal graph, and let $v_2$ and $e_2$ be the number of its vertices and edges respectively. We claim that the number of edges in $\Gamma_2$ that are not incident to $x$ is greater or equal than the number of vertices different from $x$:
\begin{equation}
\label{a3}
e_2-\deg_{\Gamma_2}(x)\geqslant v_2-1.
\end{equation}
Indeed, this follows from the fact that $\Gamma_2$ is a tree, and that edges approaching the boundary have only one vertex.

Now consider the subgraph $\Gamma$, defined as the union of $\Gamma_1$ and $\Gamma_2$. We use the notation $v$, $e$, and $f$ for the number of its vertices, edges, and faces respectively. Clearly, we have
$$
e=e_1+e_2,\qquad v=v_1+v_2-1.
$$
Combining these identities with relations~\eqref{a2}---\eqref{a3}, we obtain
$$
e-v\geqslant 2\ord_x(u)-\frac{1}{2}\deg_{\Gamma_1}(x)-1.
$$
Using the bound for the degree from Lemma~\ref{l22}, we arrive at the relation
$$
e-v\geqslant 2\ord_x(u)+\bar\chi-l-2.
$$
Finally, viewing $\Gamma$ as a subgraph in the reduced nodal graph $\bar{\mathcal N}(u)$,  we combine the last relation with the Euler inequality to obtain 
$$
f\geqslant e-v+\bar\chi-l\geqslant 2\ord_x(u)+2\bar\chi-2l-2.
$$
Since the number of faces $f$ is not greater than the number of nodal domains, we are done.
\end{proof}

\section{Proof of Theorem~\ref{t2}}
\label{proofs2}
\subsection{Structure of nodal graphs}
The proof of the theorem is based on the analysis of the equality case in~\eqref{in2}. Throughout this section we assume that $k$ is odd. Suppose the contrary to the statement. Then there exists a metric $g$ and a measure $\mu$ on the surface $M$ such that for some $k\geqslant 1$
$$
m_k(g,\mu)=2(2-\bar\chi)+2l+k=2n+1,
$$
where $n=2-\bar\chi+l+(k-1)/2$. Fix a point $x\in M$. Then by Lemma~\ref{l4}, one can find two linearly independent eigenfunctions $u_0$ and $u_1$ for $\sigma_k(g,\mu)$ such that $\ord_x(u_i)\geqslant n$, $i=0,1$. The combination of Lemma~\ref{l2} and Courant's nodal domain theorem yields
$$
2n+2\bar\chi-2l-2\leqslant k+1.
$$
Using the formula for $n$, we conclude that the inequality above becomes an equality and, in particular, $\ord_x(u_i)=n$, where $i=0,1$. Since $M$ is not homeomorphic to a disk, we also have $n\geqslant 2$.

The following lemma says that the nodal graphs of the eigenfunctions $u_i$ have a rather rigid structure. Below by the {\em nodal loop} we mean a nodal arc in the interior of $M$ that starts and ends at the same vertex. 
\begin{lemma}
\label{aux}
Let $(M,g)$ be a compact Riemannian surface with boundary, and $\mu$ be an absolutely continuous Radon measure on $\partial M$ whose density is bounded. Further, let $u$ be a non-trivial Steklov eigenfunction for the eigenvalue $\sigma_k(g,\mu)$ such that
$$
2\ord_x(u)+2\bar\chi-2l-2=k+1
$$
for some $x\in\mathcal N^2(u)$. Then the nodal graph $\mathcal N(u)$ does not contain any vertices apart from $x$ and has precisely $l+1-\bar\chi$ loops. Besides, there are no nodal arcs with both ends on the boundary.
\end{lemma}
\begin{proof}
The relation in Lemma~\ref{aux} implies that the inequalities in Lemmas~\ref{l2} and~\ref{l22} are equalities. Inspecting the proofs of these lemmas, we see that the graphs $\Gamma_1$ and $\Gamma$, defined in Section~\ref{2nd}, have the following properties:
\begin{itemize}
\item[(i)] all vertices in the subgraph $\Gamma_1$  different from $x$ have degree $2$ in $\Gamma_1$;
\item[(ii)] each face of $\Gamma_1$ contains precisely one boundary component;
\item[(iii)] the number of faces of the graph $\Gamma$ equals the number of nodal domains;
\item[(iv)] the faces of $\Gamma$ viewed as subdomains in the reduced surface $\bar M$ are simply connected (this is a consequence of the equality in the Euler inequality).
\end{itemize}
We claim that there are no vertices apart from $x$ in $\Gamma_1$. Suppose the contrary, and let $y\in\mathcal N^2(u)$ be such a vertex. Since its degree in the nodal graph is at least $4$, by property~(i) there are nodal edges incident to $y$ that do not lie in $\Gamma_1$. We may assume that these edges belong to a tree subgraph $\Gamma_0$ in a connected component of $x$ that does not intersect the graph $\Gamma_1$ except for the vertex $y$. It then also does not intersect the graph $\Gamma_2$. Now consider the subgraph $\Gamma_0\cup\Gamma$. It is straightforward to see that the difference between the number of edges and vertices for this subgraph is strictly greater than the same quantity for $\Gamma$. Now applying the argument in the proof of Lemma~\ref{l2} to the graph $\Gamma_0\cup\Gamma$ instead of $\Gamma$, we obtain a strict inequality for the number of nodal domains in Lemma~\ref{l2} and arrive at a contradiction.

Combining the claim above with the equality in Lemma~\ref{l22}, we see that the graph $\Gamma_1$ consists precisely of the vertex $x$ and $l+1-\bar\chi$ loops. In fact,  there are no vertices apart from $x$ in the graph $\Gamma=\Gamma_1\cup\Gamma_2$. Indeed, the contrary would give a strict inequality in~\eqref{a3} and in Lemma~\ref{l2}. Thus, we conclude that the connected component of $x$ in the nodal graph is precisely the graph $\Gamma$, which consists of one vertex $x$, a number of nodal arcs joining it with the boundary, and $l+1-\bar\chi$ loops.

Now we show that there are no vertices in the whole nodal graph $\mathcal N(u)$. Suppose the contrary: there is another vertex, which has to belong to a different connected component of $\mathcal N(u)$. Then this connected component viewed as a subset of $\bar M$ has to lie in a face of $\Gamma$. Denote by $\Gamma_*\subset\bar{\mathcal N}(u)$ the image of this connected component in the reduced nodal graph. We claim that $\Gamma_*$ contains a cycle. Then, since by  property~(iv) the faces of $\Gamma$ in the reduced surface are simply connected, we conclude that the number of nodal domains is strictly greater than the number of faces of $\Gamma$, and arrive at a contradiction with property~(iii). The existence of a cycle in $\Gamma_*$ follows from the existence of a subgraph whose every vertex has degree at least two; such a subgraph then has to contain a cycle, see Section~\ref{prems}. Indeed, if the connected component does not have edges approaching the boundary, then it can be taken as such a subgraph. If otherwise, by property~(ii) its edges can approach only one boundary component; that is, the one that lies in the same face of $\Gamma$.  If the boundary component vertex has degree at least two in $\Gamma_*$, then every vertex of $\Gamma_*$ has degree at least two, and we are done.  If it has degree one, then we remove the corresponding incident edge and the boundary component vertex from $\Gamma_*$. The result is a non-trivial subgraph of $\Gamma_*$ whose every vertex has degree at least two. 

Finally, the statement that there are no nodal arcs with both ends on the boundary is a direct consequence of property~(ii). 
\end{proof}

\subsection{Isotopy of nodal graphs}
\label{isotopy}
Since $M$ is not homeomorphic to a disk, we have $l+1-\bar\chi\geqslant 1$, and by Lemma~\ref{aux} the nodal graph $\mathcal N(u_0)$ has at least one loop. Besides, it also has 
$$
2\ord_x(u_0)-2(l+1-\bar\chi)=k+1\geqslant 2
$$
nodal arcs incident to $x$ and approaching the boundary. Now we explain an isotopy argument, showing how the existence of at least one loop and at least one arc in $\mathcal N(u_0)$ leads to a contradiction. Following the idea in~\cite{HoNa1,HoNa}, we construct an isotopy of the nodal graph $\mathcal N(u_0)$ to itself that deforms a nodal loop to a nodal arc. We start with the family of eigenfunctions
\begin{equation}
\label{deform}
u_t=u_0\cos nt+u_1\sin nt,\qquad\text{where }t\in[0,\pi].
\end{equation}
It is straightforward to see that the family $\mathcal N(u_t)$ defines an {\em isotopy of nodal graphs} in the sense of~\cite{HoNa}; that is, a family of graphs such that every nodal arc deforms smoothly among embedded arcs in the interior of $M$, and vertices do not change their multiplicity. More precisely, all graphs $\mathcal N(u_t)$ have only one vertex at the same point $x$, and the number of nodal domains of $u_t$ is maximal, that is equal to $k+1$. The fact that the arcs deform smoothly follows from the implicit function theorem. We claim that under such an isotopy loops deform into loops. Indeed, by Lemma~\ref{aux} the number of loops in $\mathcal N(u_t)$ is constant and is equal to $l+1-\bar\chi$. The claim follows from the fact that the property of a nodal arc to be a loop is open in time $t$.

By Proposition~\ref{BersTheorem} we can assume that the eigenfunctions $u_0$ and $u_1$ in geodesic polar coordinates centered at $x$ have the form
\begin{equation*}
\begin{split}
u_0 = r^n\sin n\varphi + O(r^{n+1});\\
u_1 = r^n\cos n\varphi + O(r^{n+1}).
\end{split}
\end{equation*}
The deformation $u_t$ then takes the form
$$
u_t = r^n\sin n(\varphi+t) + O(r^{n+1}).
$$ 
The nodal set of $u_0$ around $x$ is diffeomorphic to the union of $2n$ straight rays meeting at the origin; they satisfy the equations $\varphi=\varphi_j:=(j\pi)/n$, where $j=0,\ldots,2n-1$. Performing a rotation in polar coordinates $(r,\varphi)$, we may assume that the rays with the angles $\varphi_0$ and $\varphi_j$ for some $j$ are contained in a nodal loop and a nodal arc respectively. The nodal sets of $u_0$ and $u_t$ at $t=\varphi_j$ coincide, and the deformation given by~\eqref{deform} with $t\in [0,\varphi_j]$ is an isotopy of the nodal graph $\mathcal N(u_0)$ to itself. This isotopy transforms the ray $\varphi_0$ to the ray $\varphi_j$. Thus, we see that a loop transforms to an arc, and arrive at a contradiction.
\qed

\section{Asymptotic multiplicity bounds}
\subsection{Proof of Theorem \ref{as1}} 
Consider a ``weighted'' Dirichlet-to-Neumann operator on $\partial M$ that sends
$$
C^\infty(\partial M)\ni u\longmapsto\rho^{-1} \frac{\partial \hat u}{\partial\nu}\in C^\infty(\partial M),
$$ 
where $\hat u$ denotes the unique harmonic extension  of $u$ into $M$. When $\rho$ is smooth and positive, it defines a self-adjoint elliptic pseudo-differential operator of the first order whose eigenvalues are precisely the Steklov eigenvalues, see~\cite[pp. 37-38]{Ta} and~\cite{Ros}. Let $N(\lambda)$ be its eigenvalue counting function; it equals the number of eigenvalues counted with multiplicity that is strictly less than a positive $\lambda$. By H\"ormander's theorem~\cite{Ho}, see also~\cite{Shubin}, the function $N(\lambda)$ satisfies the following asymptotics (Weyl's law): 
\begin{equation}
\label{Weyl}
N(\lambda)=\frac{\lambda}{2\pi} \int_{\partial M} \rho(s)\, ds_g  + R(\lambda),
\end{equation}
where $R(\lambda)$ is a bounded quantity in $\lambda>0$. Using this formula, we obtain
\begin{multline*}
m_k(g,\mu)=\lim_{\epsilon \to 0} N(\lambda_k+\varepsilon)-N(\lambda_k)=\\ \lim_{\varepsilon \to 0} \frac{\varepsilon}{2\pi} \, \int_{\partial M} \rho(s)\, ds_g + 
R(\lambda_k+\epsilon) -R(\lambda_k)\leqslant 2 \sup\abs{R(\lambda)}.
\end{multline*}
Thus, the multiplicity $m_k(g,\mu)$ is indeed bounded, and the theorem is proved.\qed

It is interesting to know up to what extent the bound on $m_k(g,\mu)$ depends on a metric and a boundary measure; in particular, whether there exists a universal constant (possibly depending on the genus of $M$) for which Theorem~\ref{as1} holds.  

\subsection{Proof of Proposition~\ref{as2}} By the uniformisation theorem, we may assume that $M$ is a unit disk and the metric $g$ on $M$ is conformal to the Euclidean metric $g_{\mathit Euc}$. Since the Dirichlet energy is conformally invariant, by the variational principle we see that the Steklov eigenvalues of $(M,g)$ with a weight function $\rho$ coincide with the Steklov eigenvalues of $(M,g_{\mathit Euc})$ with the a new weight function $\rho_0$ that depends on $\rho$ and the values of $g$ on $\partial M$ only.
By the results in~\cite{Ros, Ed} the latter satisfy the following refinement of Weyl's asymptotic formula:
\begin{equation}
\label{ros}
\sigma_{2k}=\frac{2\pi\, k}{\int_{\partial M} \rho_0(s) \, ds} +o(k^{-\infty}), \quad \sigma_{2k+1}=\frac{2\pi\, k}{\int_{\partial M} \rho_0(s) \, ds} +o(k^{-\infty}),
\end{equation}
as $k \to \infty$. Thus, we conclude that for a large $k$ the multiplicity of the eigenvalue $\sigma_k$ is two at most. \qed

We end with two remarks. First, note that for a Euclidean disk all non-zero eigenvalues have multiplicity two, and therefore, the statement of Proposition~\ref{as2} is sharp.  Second, the hypotheses of Proposition~\ref{as2} on the smoothness of $\partial M$ and $\rho>0$ are essential for  the asymptotic formula~\eqref{ros} to hold.  Even for domains with piecewise smooth boundaries the asymptotic properties of the spectrum may be quite different. In particular, by a direct computation one can show that formulas~\eqref{ros} fail for a  square: for a large $k$  the Steklov spectrum of a square is the union of quadruples of eigenvalues,  such that in each quadruple the eigenvalues are $o(k^{-\infty})$-close \cite{Gir}. However, no counterexample to Proposition \ref{as2} is known for simply-connected surfaces with non-smooth boundaries, and it would be interesting to understand whether the result holds in this case as well.

\section{Other boundary value problems}
\label{other}
\subsection{Eigenvalue problems with homogeneous boundary conditions}
\label{sec:robin}
The method used to prove the first inequality in Theorem~\ref{t1} relies only on  Courant's nodal domain theorem and the behaviour of eigenfunctions in the interior of $M$; it largely disregards their behaviour on the boundary. The purpose of this section is to show that it applies to rather general boundary value problems.

Let $(M,g)$ be a compact Riemannian surface with boundary and $L=(-\Delta_g)+V$ be a Schr\"odinger operator, where $V$ is a smooth potential.  Denote by $B$ a boundary differential operator of the form
\begin{equation}
\label{robin_boundary}
Bu=au+b\frac{\partial u}{\partial v},
\end{equation}
where $a$ and $b$ are bounded functions on $\partial M$ that do not vanish simultaneously. We consider the following eigenvalue problem
\begin{equation}
\label{eq1}
Lu=\lambda u\quad\text{in }M,\quad\text{and}\quad Bu=0\quad\text{on }\partial M.
\end{equation}
It is often referred to as the Robin boundary value problem; the Dirichlet and Neumann problems are its special cases. By
$$
\lambda_0 < \lambda_1\leqslant\ldots\lambda_k\leqslant\ldots
$$
we denote the corresponding eigenvalues, where $\lambda_0$ is the bottom of the spectrum.

The following statement gives a bound for the eigenvalue multiplicities of problem~\eqref{eq1} that is independent of a Schr\"odinger operator $L$ and, more interestingly, of a boundary operator $B$.
\begin{prop}
\label{robin}
Let $M$ be a compact Riemannian surface with a non-empty boundary. Then for any Schr\"odinger operator $L$  and any Robin boundary operator $B$ given by~\eqref{robin_boundary}  the multiplicity $m_k$ of an eigenvalue $\lambda_k$ corresponding to problem~\eqref{eq1} satisfies the inequality
\begin{equation}
\label{in11}
m_k\leqslant 2(2-\bar\chi)+2k+1,
\end{equation}
for all $k=1,2,\ldots,$ where $\bar\chi=\chi+l$, and $\chi$ and $l$ stand for the Euler-Poincar\'e number and the number of boundary components of $M$ respectively. 
\end{prop}
For the Dirichlet and Neumann eigenvalues on simply connected domains, the estimate~\eqref{in11} is due to~\cite{Na}. In this case, the bound is sharp for $k=1$. The method used in~\cite{Na} does not extend to arbitrary Riemannian surfaces with boundary.

\subsection{Details on the proof}
We explain how the arguments and results in Sections~\ref{prems} and~\ref{proofs} could be extended to prove Proposition~\ref{robin}. First, Proposition~\ref{BersTheorem} holds for solutions of second order elliptic differential equations with smooth coefficients. In particular, it holds for eigenfunctions of problem~\eqref{eq1}. Thus, the nodal set of an eigenfunction has a similar graph structure. These eigenfunctions also enjoy Courant's nodal domain theorem, see~\cite{CH}, and the arguments in Section~\ref{proofs} show that their nodal graphs are finite. A version of Proposition~\ref{BersTheorem} also implies that the statement of Lemma~\ref{l4} holds for solutions of general second order elliptic equations, cf.~\cite[Lemma~4]{Na}. The rest of the proof of the first inequality in Theorem~\ref{t1} carries over without changes.

Finally, let us mention that inequality~\eqref{in11} is also valid for eigenvalue problems with mixed boundary conditions. In addition, one can also allow non-smooth boundaries as long as the eigenvalue problem remains well-posed and Courant's nodal domain theorem holds. 

\subsection*{Acknowledgments} The authors would like to thank Alexandre Girouard for useful comments. MK is grateful to his supervisor Alexei Penskoi for inspiring discussions on spectral geometry. The research of IP was partially supported by NSERC, FQRNT and the Canada Research Chairs program.  The research of MK was partially supported by the Simons Fellowship and the Dobrushin Fellowship. Part of this project was completed while GK and MK were visiting the Centre de recherches math\'emati\-ques in Montr\'eal. Its hospitality is gratefully acknowledged.

\end{document}